\documentclass[preprint,review,12pt]{elsarticle}
\usepackage[T1]{fontenc}
\usepackage[latin9]{inputenc}
\usepackage{float}
\usepackage{mathtools}
\usepackage{url}
\usepackage{amsmath}
\usepackage{amsthm}
\usepackage{amssymb}
\usepackage{graphicx}

\makeatletter
\theoremstyle{plain}
\newtheorem{thm}{\protect\theoremname}
\theoremstyle{remark}
\newtheorem{rem}{\protect\remarkname}

\usepackage{url}	 % for hyphenizing emails and urls
\usepackage{cite}

\usepackage{graphicx}
\usepackage{xcolor}
\usepackage{import}

\makeatother

\providecommand{\remarkname}{Remark}
\providecommand{\theoremname}{Theorem}

\begin{document}

\begin{frontmatter}{}

\title{A Robust Unscented Transformation for Uncertain Moments}

% Copyright information:
\tnotetext[t1]{\copyright 2019. This manuscript version is made available under the CC BY-NC-ND 4.0 license \url{http://creativecommons.org/licenses/by-nc-nd/4.0/}. 
This article has been accepted for publication in a future issue of the Journal of Franklin Institute, but has not been fully edited. Content may change prior to final publication. 
The final version of record is available at \url{https://doi.org/10.1016/j.jfranklin.2019.02.018}.}

\author[unb]{Hugo T. M. Kussaba\corref{cor1}}

\ead{kussaba@lara.unb.br}

\author[unb]{João Y. Ishihara}

\author[unb]{Leonardo R. A. X. Menezes}

\cortext[cor1]{Corresponding author}

\address[unb]{Department of Electrical Engineering, University of Brasília --
UnB, 70910-900, Brasília, DF, Brazil}
\begin{abstract}
This paper proposes a robust version of the unscented transform (UT)
for one-dimensional random variables. It is assumed that the moments
are not exactly known, but are known to lie in intervals. In this
scenario, the moment matching equations are reformulated as a system
of polynomial equations and inequalities, and it is proposed to use
the Chebychev center of the solution set as a robust UT. This method
yields a parametrized polynomial optimization problem, which in spite
of being NP-Hard, can be relaxed by some algorithms that are proposed
in this paper.
\end{abstract}

\end{frontmatter}{}

\section{Introduction}

In numerous problems of statistics and stochastic filtering, one is
often interested in calculating the posterior expectation of a continuous
random variable $X$ that undergoes a nonlinear transform $f$, viz.:
\begin{equation}
\textrm{E}\left\{ f(X)\right\} =\int_{\mathbb{R}}f(\xi)p_{X}(\xi)d\xi.\label{eq:posterior_distribution}
\end{equation}

It is not always possible to have a closed-form expression for this
integral in terms of elementary functions: if this integral does not
satisfy the hypothesis of Liouville's theorem (see for instance, Section~12.4
of \citep{GVL:92}), then the antiderivative of this integral cannot
be expressed in terms of elementary functions. Thus, instead of using
analytical methods to calculate (\ref{eq:posterior_distribution}),
in many situations numerical methods must be employed. 

A common way to numerically calculate (\ref{eq:posterior_distribution})
is by using the Monte Carlo integration method \citep{PP:02}, which
is a stochastic sampling technique: by taking a sufficiently large
number of samples of the random variable $X$, one can approximate
the probability density function $p_{X}$ and obtain an estimate for
(\ref{eq:posterior_distribution}). However, this method can be very
demanding computationally, since it frequently employs several thousands
of simulations to obtain the statistics of the final result.

Another way to calculate (\ref{eq:posterior_distribution}) with less
computational burden than Monte Carlo integration method is the technique
of Unscented Transform (UT). Originally proposed for the problem of
extending the Kalman filter for nonlinear dynamical systems \citep{JUD:95},
this method has also been applied in several problems of engineering,
such as in the analysis of the sensitivity of antennas \citep{MSSTC:10}
 and in circuit optimization \citep{CCDBMKAR:11}. 

Different from the Monte Carlo integration method, the UT is a deterministic
sampling technique: in place of choosing random points to approximate
$p_{X}$, points known as sigma points are deterministically selected
to capture the statistics of $p_{X}$. This is accomplished by generating
a discrete distribution $p'_{X}$ having the same first and second
(and possibly higher) moments of $p_{X}$. The mean, covariance and
higher moments of the transformed ensemble of sigma points can then
be computed as the estimate of the nonlinear transformation $f$ of
the original distribution \citep{JUD:95}, \citep{MIBV:15}. Thus,
at least the first moments of $p_{X}$ must exist and be exactly known.
In several circumstances, however, this is not valid: for instance,
it may be the case that the distribution does not even have the first
moment (one example is the Cauchy distribution \citep{Kri:10}). In
the case that the moments can be assumed to exist, it is usual that
they are not precisely known, but it is still possible to have upper
and lower bounds for the exact value of moments from statistical experiments
or by the use, for instance, of Chebychev's inequality \citep{PP:02}.

In this paper it is designed a technique for generating sigma points
when the exact value of the moments are not known, but upper and lower
bounds for the unknown moments are known. In this case, the moment
matching equations of UT are no longer just a system of polynomial
equations but a system of polynomial equations with polynomial inequalities.
Furthermore, since this system can have more than one solution, it
is possible to choose a set of sigma points which minimizes a given
cost function by formulating the problem as a polynomial optimization
problem (POP). Although the solution to this problem is in general
computationally infeasible \citep{MK:87}, by using Lasserre's hierarchy
of semidefinite programming relaxations one can approximate the solution
of the original POP problem by the solution of a computationally feasible
convex optimization problem\citep{Las:01,Las:09}.

The main contribution of this paper is the introduction of the concept
of UT robustness in the sense of exploiting the upper and lower bounds
for moments. A robust UT is proposed in \citep{Van:01} but robustness
has different meaning and it is achieved by matching precisely known
high order moments. It is important to note that Lasserre's hierarchy
of relaxations was applied in earlier investigations of the moment
matching problem \citep{MP:13}, but its use was limited to polynomial
equations while here polynomial inequalities are also taken into account.

This paper is organized as follow. Some preliminaries for the robust
UT are presented in Section~\ref{sec:Unscented-Transform}. The robust
UT itself is detailed in Section~\ref{sec:Proposed_Technique}. Details
about the computation of robust sigma points is presented in Section~\ref{sec:Computation-of-robust-sigma-points}.
The computation of UT transform is presented in Section~\ref{sec:Computation-of-UT}.
Finally, the conclusion is presented in Section~\ref{sec:Conclusion}.

\section{Unscented Transform \label{sec:Unscented-Transform}}

The rationale behind the unscented transformation is that it is easier
to calculate a moment for a discrete distribution than a continuous
one. In fact, to compute the posterior distribution of a random scalar
variable $X$ with distribution $p_{X}$ by a function $f$, one must
calculate the integral (\ref{eq:posterior_distribution}) while, in
other hand, if $Z$ is a discrete random variable with $m+1$ atoms
$z_{i}$ and distribution $p_{Z}$, one only needs to calculate the
sum 
\begin{equation}
E\left\{ f(Z)\right\} =\sum_{i=1}^{m+1}f(z_{i})p_{Z}(z_{i}).\label{eq:posterior_discrete}
\end{equation}

Henceforth, if it is possible to choose the atoms and their weights
of an adequate $p_{Z}$ to approximate $p_{X}$, then the value $E\{f(Z)\}$
would be a good approximation to $E\{f(X)\}$, but with (\ref{eq:posterior_discrete})
being easier to compute than (\ref{eq:posterior_distribution}). 
Thus, the principle behind the UT is to approximate the continuous
distribution $p_{X}$ by the discrete distribution $p_{Z}$ by equating
the first $m$ moments of these distributions. In other words, the
following equations must be satisfied:
\begin{align}
E\{X^{k}\} & =E\{Z^{k}\}\\
 & =\sum_{i=1}^{m+1}z_{i}^{k}p_{Z}(z_{i}),\ k=1,\ldots,m,\label{eq:moment_matching_equation_polynomial}
\end{align}
where is supposed the access to $E\{X^{k}\}$ for $k=1,\ldots,m$. 

Given $E\{X^{k}\}$ it is always possible to find $z_{i}$ and $p_{Z}(z_{i})$,
$i=1,\ldots,m+1$ (which will be called sigma point and its weight,
respectively) satisfying (\ref{eq:moment_matching_equation_polynomial}).
In fact, as it will be stated next, there is at least one solution
for the equations
\begin{equation}
E\left\{ g_{k}(Z)\right\} =E\left\{ g_{k}(X)\right\} ,\ k=1,\ldots,m,\label{eq:moment_matching_equation_generic}
\end{equation}
where $g_{k}:\mathbb{R}\rightarrow\mathbb{R}$ are continuous functions,
$g_{k}\not\equiv g_{j}$ for $k\ne j$ and from which (\ref{eq:moment_matching_equation_polynomial})
is a particular case with $g_{k}(x)=x^{k}$. That (\ref{eq:moment_matching_equation_generic})
has at a least one solution is stated next in Theorem~\ref{thm:existence_of_solution}.

\begin{thm}
\label{thm:existence_of_solution}Consider that the $m$ moments $E\{g_{k}(X)\}$,
$k=1,\ldots,m$, are given. The system~(\ref{eq:moment_matching_equation_generic})
in terms of variables $z_{i}$ and $p_{Z}(z_{i})$ has at least one
solution with at most $m+1$ sigma points.\footnote{The proof of the theorem is known in the mathematical literature in
the context of the Caratheodory's theorem. Since it is less known
in the context of UT literature, the proof is presented here for easy
reference. See, e.g., \citep{DS:97}.} 
\end{thm}
\begin{proof}
Since the point $P=\left(E\left\{ g_{1}(X)\right\} ,\ldots,E\left\{ g_{m}(X)\right\} \right)$
belongs to the convex hull of the set $G=\left\{ \left(g_{1}(x),\ldots,g_{m}(x)\right)\mid x\in\mathbb{R}\right\} $,
then Caratheodory's theorem \citep[Theorem 1.3.6]{HJL:01}  gives
that $P$ can be written as a convex combination of at most $m+1$
points in $G$. Thus, one has that there exist $\theta_{i}\ge0$ and
$z_{i}\in\mathbb{R}$ such that
\begin{equation}
E\{g_{k}(X)\}=\sum_{i=1}^{m+1}\theta_{i}g_{k}(z_{i}),\ k=1,\ldots,m\label{eq:existence_eq1}
\end{equation}
and 
\begin{equation}
\sum_{i=1}^{m+1}\theta_{i}=1.\label{eq:existence_eq2}
\end{equation}

Taking $Z$ to be the discrete random variable with probability distribution
given by
\[
p_{Z}(k)=\begin{cases}
\theta_{i}, & \text{if }k=z_{i},\\
0, & \text{otherwise,}
\end{cases}
\]
one has that equations (\ref{eq:existence_eq1}) and (\ref{eq:existence_eq2})
are exactly the equations given by (\ref{eq:moment_matching_equation_generic}).
\end{proof}
It is important to note that Theorem~\ref{thm:existence_of_solution}
only states the existence of sigma points satisfying (\ref{eq:moment_matching_equation_generic}),
but it does not give any conclusion about uniqueness. In fact,  many
solutions are possible, and the choice of an adequate set of sigma
points has been investigated thoroughly in the literature \citep{MIBV:15,RYB:18}.
\begin{rem}
Specifically for the first moments of $X$ focused in this paper,
i.e. the case that $g_{k}(x)=x^{k}$, $k=0,\ldots,m+1$, one can also
see (\ref{eq:moment_matching_equation_polynomial}) as a Gaussian
quadrature integration scheme with $p_{X}$ being the weighting function
and $z_{i}$ and $\omega_{i}:=p_{Z}(z_{i})$, $i=1,\ldots,m+1$, being
respectively the nodes and their weights in the quadrature formula.
Thus, depending on the probability density function $p_{X}$, the
sigma points can be readily calculated as the roots of an orthogonal
polynomial \citep[Section 4.6]{PTVF:07}. For instance, if $p_{X}$
is a normal distribution, then the sigma points are the roots of a
Hermite polynomial.

One can note that if the only intention were to find some discrete
random variable such that $E\{f(X)\}=E\{f(Z)\}$, it is always possible
to find a $Z$ variable with two sigma points. In fact, in Theorem~\ref{thm:existence_of_solution},
consider $m=1$ and $f=g_{1}$. However, this would imply the knowledge
of the function $f$. In the UT reasoning, it is sought a greater
number of sigma points in order that the approximation $E\{f(X)\}=E\{f(Z)\}$
be valid for a greater number of functions $f$. In \citep{JU:04}
it is shown that the approximation is good for any function $f$ which
can be well approximated by its $m$-order Taylor representation.
\end{rem}
On top of that, the larger is $m$, the more precise is the approximation
for $E\{f(X)\}$: an estimate for the error $\int_{a}^{b}p_{X}(x)f(x)\,dx-\sum_{i=1}^{m}\omega_{i}f(z_{i})$
is given by
\[
\frac{f^{(2m)}(\xi)}{(2m)!}\int_{a}^{b}p_{X}(x)h_{m}^{2}(x)\,dx,
\]
where $\xi\in\left(a,b\right)\subset\mathbb{R}$ and $h_{m}$ is the
associated monic orthogonal polynomial of degree $m$ associated to
$p_{X}$ \citep[Theorem 3.6.24]{SB:02}. As a matter of fact, the
Gauss quadrature computed integral for $E\{f(X)\}$ is exact for all
$f(x)$ that are polynomials of degree less or equal than $2m-1$
\citep[pp. 172-175]{SB:02}. 

Since Theorem \ref{thm:existence_of_solution} takes in account generic
functions $g_{k}$, one can analyze very general moment settings (for
instance, the case of fractional moments is worked in \citep{CHL:12}).

Summing up, one can note that the basic assumption for the UT theory
until now is that the moments are precisely known. However, this assumption
can be too strong in practical situations where the moments are estimated
from experiments and then, their values are only known to be in some
intervals. In this work, we deal with the scenario of unknown moments
and propose to calculate a robust set of sigma-points in the sense
of minimizing the worst possible error between this robust choice
of sigma-points and the sigma-points computed by using the real values
of the moments.

\section{Robust UT\label{sec:Proposed_Technique}}

To motivate the proposed robust UT consider a normally distributed
random variable $X$ with mean $\mu$ and variance $V$. In the case
where $\mu$ and $V$ are given and $m=2$, it is known that a set
of sigma points are given by \citep{MIBV:15}
\begin{equation}
\begin{array}{ccc}
z_{1} & = & \mu-\sqrt{3V},\\
z_{2} & = & \mu,\\
z_{3} & = & \mu+\sqrt{3V},
\end{array}\label{eq:sigma_points_for_normal}
\end{equation}
with its weights given respectively by $\omega_{1}=\frac{1}{6}$,
$\omega_{2}=\frac{2}{3}$ and $\omega_{3}=\frac{1}{6}$. 

Suppose now that the value of $V$ is not exactly known, but an upper
bound $\overline{V}$ and a lower bound $\underline{V}$ for $V$
is known, i.e. $V\in\left[\underline{V},\overline{V}\right]$. A naive
approach for choosing the sigma points would be to use the mean value
between $\overline{V}$ and $\underline{V}$ for $V$ in (\ref{eq:sigma_points_for_normal}).
In this case, the sigma points $z_{1}$ and $z_{3}$ would be given
by
\begin{align*}
z_{1}= & z_{1}^{M}\coloneqq\mu-\sqrt{3(\underline{V}+\overline{V})/2},\\
z_{3}= & z_{3}^{M}\coloneqq\mu+\sqrt{3(\underline{V}+\overline{V})/2},
\end{align*}
and $z_{2}$, and the weights $\omega_{1}$, $\omega_{2}$ and $\omega_{3}$
unchanged.

As Figure~\ref{fig:choice_of_sigma_points} illustrates, this can
be a pessimistic choice for the sigma points, since if, for example,
the real value of $V$ is in fact nearer of $\bar{V}$, a better choice
of sigma points would be nearer to the point $\left(z_{1},z_{3}\right)=\left(\mu-\sqrt{3\overline{V}},\mu+\sqrt{3\overline{V}}\right)$.
In fact, as can be seen in Figure~\ref{fig:choice_of_sigma_points},
a preferable choice for the sigma points would be at the center of
the region: 
\begin{align*}
z_{1}= & z_{1}^{C}\coloneqq\mu-(\sqrt{3}/2)(\sqrt{\underline{V}}+\sqrt{\overline{V}}),\\
z_{3}= & z_{3}^{C}\coloneqq\mu+(\sqrt{3}/2)(\sqrt{\underline{V}}+\sqrt{\overline{V}}).
\end{align*}
This choice is precisely the Chebychev center\footnote{There are two non-equivalent definitions of Chebychev center of a
bounded set with non-empty interior: the first definition is the center
of the minimal-radius ball enclosing this set, and the second is the
center of the largest incribed ball in this set \citep{BV:04}. In
this paper, only the first definition will be used.} of the set of possibles sigma points given that $V\in\left[\underline{V},\overline{V}\right]$.
It has the property of having the minimum worst possible error between
the chosen sigma point set and the sigma point set corresponding to
the true value of $V$.

\begin{figure}[H]
\begin{centering}
\includegraphics[scale=0.8]{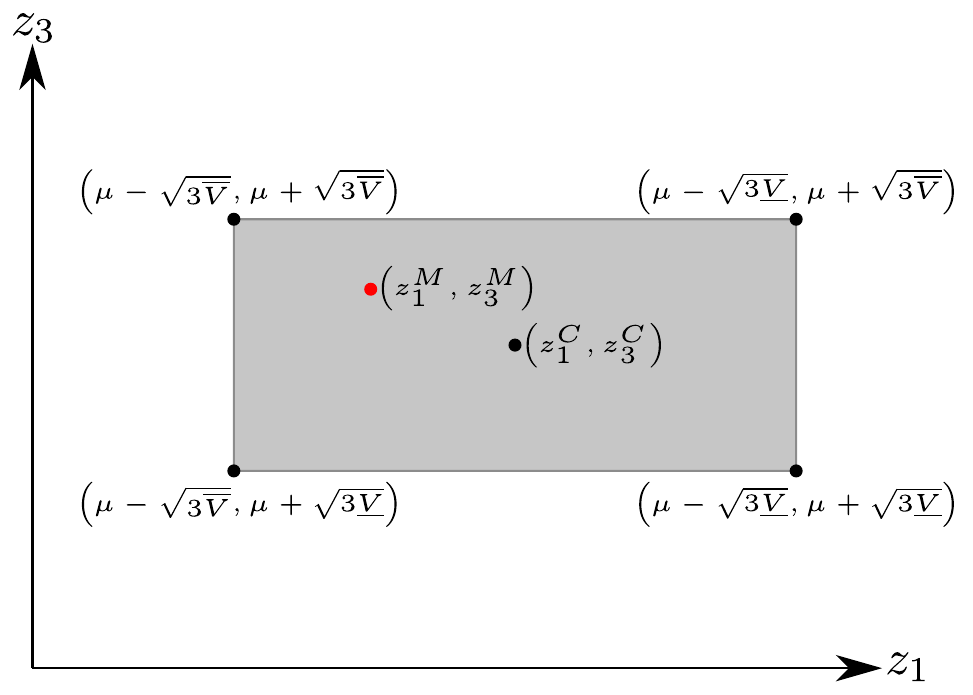}
\par\end{centering}
\caption{Possible choices of sigma points. \label{fig:choice_of_sigma_points}}
\end{figure}

Consider now a general scenario where some of the moments are not
known, but upper and lower bounds for them are. In this case, the
set of possible sigma points and their weights would be in the semialgebraic
set $S\subset\mathbb{R}^{2m+2}$ of elements $x=(z_{1},\ldots,z_{m+1},\omega_{1},\ldots,\omega_{m+1})$,
defined as solution of the system
\begin{equation}
S:\qquad\begin{array}{cccc}
\omega_{j} & \ge & 0, & j\in\{1,\ldots,m\},\\
\sum_{i=1}^{m+1}\omega_{i} & = & 1,\\
\sum_{i=1}^{m+1}z_{i}^{k_{1}}\omega_{i} & = & E\{X^{k_{1}}\}, & k_{1}\in\mathcal{K}_{1},\\
\sum_{i=1}^{m+1}z_{i}^{k_{2}}\omega_{i} & \le & u_{k_{2}}, & k_{2}\in\mathcal{K}_{2},\\
\sum_{i=1}^{m+1}z_{i}^{k_{2}}\omega_{i} & \ge & \ell_{k_{2}}, & k_{2}\in\mathcal{K}_{2},
\end{array}\label{eq:moment_matching_equation_with_uncertainty}
\end{equation}
where $\mathcal{K}_{1}$ (resp. $\mathcal{K}_{2}$) is the set of
indexes $k$ for which the $k$-moment is known (resp. unknown), $\mathcal{K}_{1}\cup\mathcal{K}_{2}=\{1,\ldots,m\}$
and $u_{k_{2}}$ and $\ell_{k_{2}}$ are respectively known upper
and lower bounds for the $k_{2}$-moment. 

Although Theorem~\ref{thm:existence_of_solution} guarantees a solution
for (\ref{eq:moment_matching_equation_with_uncertainty}), this solution
may not be unique. On that account, there is more than one set of
sigma points and corresponding weights, and it is natural to ask which
is a good choice of sigma points and weights. Based on the previous
example, a so-called robust choice would be the Chebychev center of
the semialgebraic set defined by (\ref{eq:moment_matching_equation_with_uncertainty})
since this choice minimizes the worst possible error between the chosen
sigma point set and the sigma point set corresponding to the possible
true values of the moments of $X$.

Different from the previous example, however, in this generic scenario
it may be the case that an analytic formula to express the sigma points
as the function of the moments of the random variable $X$ is not
available, and the following optimization problem must be solved to
find the Chebychev center:
\begin{equation}
(z_{1}^{C},\ldots,z_{m+1}^{C},\omega_{1}^{C},\ldots,\omega_{m+1}^{C})=\underset{\hat{x}\in\mathbb{R}^{2m+2}}{\arg\min}\ \underset{x\in S}{\max}\|x-\hat{x}\|^{2},\label{eq:Chebychev_center_optimization}
\end{equation}
where $S$ is the solution set for (\ref{eq:moment_matching_equation_with_uncertainty})
and $x=(z_{1},\ldots,z_{m+1},\omega_{1},\ldots,\omega_{m+1})$.

It is important to note that the optimization problem in (\ref{eq:Chebychev_center_optimization})
is not always guaranteed to have a solution, since the set $S$ may
not be bounded (if, for example, $\omega_{1}$ is zero, then $z_{1}$
can take any value). Nevertheless, it still makes sense to use the
Chebychev center of a large bounded subset of $S$ to try to choose
a set of sigma points minimizing the worst possible estimation error.
In fact, for any (bounded or unbounded) set $S$, one can take an
$\varepsilon>0$ and consider the Chebychev center of the semi-algebraic
set $S_{\varepsilon}\subseteq S$ defined as the set of points 
\[
x=(z_{1},\ldots,z_{m+1},\omega_{1},\ldots,\omega_{m+1})
\]
satisfying
\begin{equation}
S_{\varepsilon}:\qquad\begin{array}{cccc}
\omega_{j} & \ge & \varepsilon, & j\in\{1,\ldots,m\},\\
\sum_{i=1}^{m+1}\omega_{i} & = & 1,\\
\sum_{i=1}^{m+1}z_{i}^{k_{1}}\omega_{i} & = & E\{X^{k_{1}}\}, & k_{1}\in\mathcal{K}_{1},\\
\sum_{i=1}^{m+1}z_{i}^{k_{2}}\omega_{i} & \le & u_{k_{2}}, & k_{2}\in\mathcal{K}_{2},\\
\sum_{i=1}^{m+1}z_{i}^{k_{2}}\omega_{i} & \ge & \ell_{k_{2}}, & k_{2}\in\mathcal{K}_{2}.
\end{array}\label{eq:approx_moment_matching_equation_with_uncertainty}
\end{equation}
As $\varepsilon$ decreases, $S_{\varepsilon}$ covers a larger part
of $S$. Fig.~\ref{fig:approximation-of-semi-algebraic-set} illustrates
how the set $S_{\varepsilon}$ approximates the set $S$ by choosing
a sufficiently small $\varepsilon>0$.

While the Chebychev center of $S$ may not exist, the next theorem
assures the existence of the Chebychev center of $S_{\varepsilon}$
for any $\varepsilon>0$.
\begin{thm}
\label{thm:semi-algebraic_set_is_bounded} If $m\ge2$ and $\varepsilon>0$,
then the optimization problem
\begin{equation}
(z_{1}^{C},\ldots,z_{m+1}^{C},\omega_{1}^{C},\ldots,\omega_{m+1}^{C})=\underset{\hat{x}\in\mathbb{R}^{2m+2}}{\arg\min}\ \underset{x\in S_{\varepsilon}}{\max}\|x-\hat{x}\|^{2},\label{eq:approx_Chebyshev_center_optimization}
\end{equation}
 has a solution.
\end{thm}
\begin{proof}
It suffices to prove that the set $S_{\varepsilon}$ is bounded. The
variables $\omega_{i}$, $i=1,\ldots,m+1$ are all inside the $m+1$
dimensional simplex and thus bounded. Since
\[
\sum_{i=1}^{m+1}z_{i}^{2}\omega_{i}\le u_{2},
\]
and $\omega_{i}>0$, it is impossible for any variable $z_{i}$ to
grow without bound.
\end{proof}
\begin{rem}
It is interesting to note that alternatively, instead of the compact
(closed and bounded) set $S_{\varepsilon}$ one could consider the
set $\hat{S}$ defined as the set $S$ up to the first inequalities
replaced by a strict inequality (that is, $\omega_{j}\ge0$ is replaced
by $\omega_{j}>0$, $j\in\{1,\ldots,m\}$). This set is indeed bounded
as $S_{\varepsilon}$. However, strict inequalities are in general
not well-handled by numeric solvers. 
\end{rem}
\begin{figure}
\centering{}\includegraphics[scale=0.6]{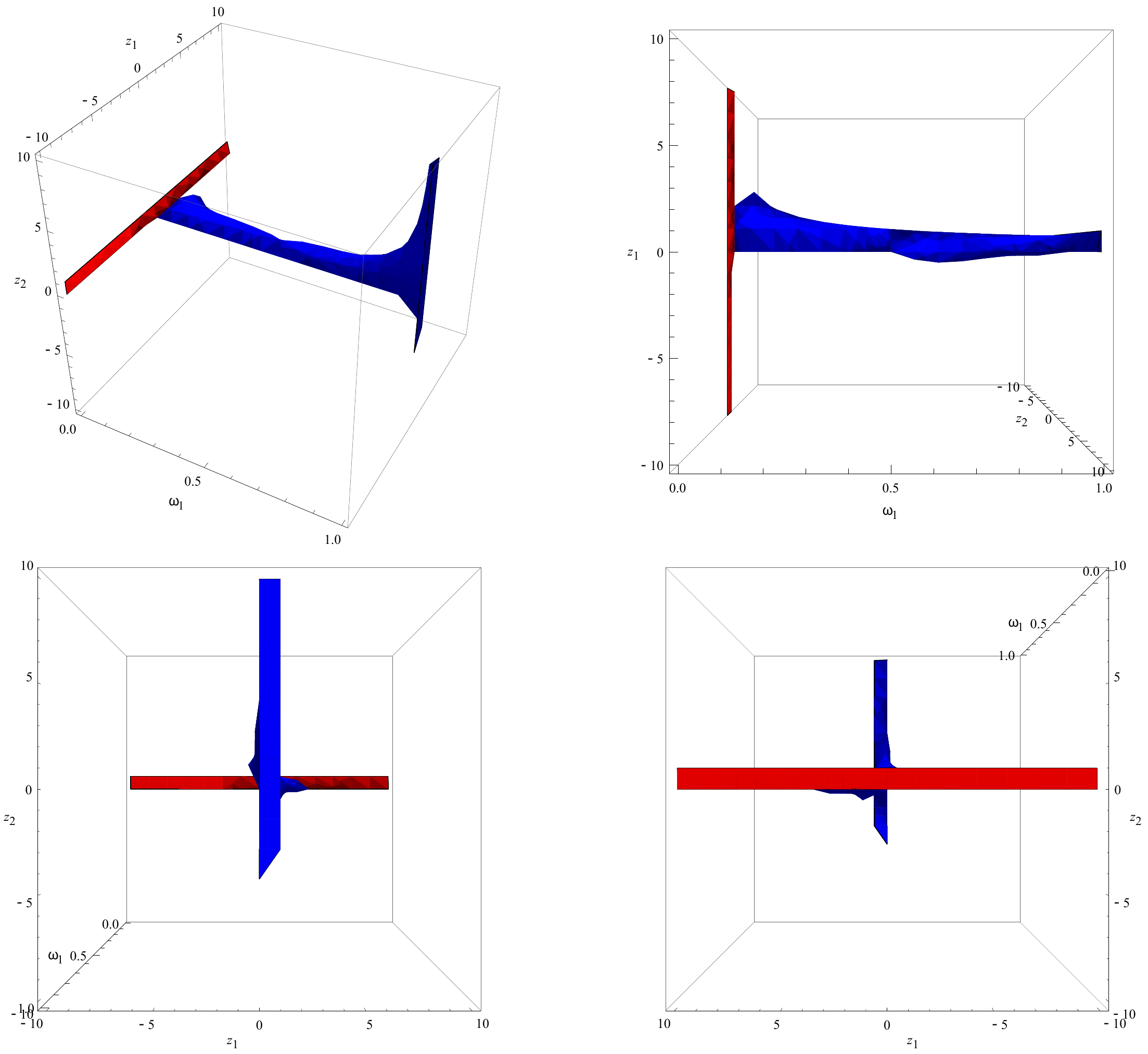}\caption{The regions $S\setminus S_{\varepsilon}$ (red) and $S_{\varepsilon}$
(blue) are plotted in various viewpoints to illustrate the approximation
of the semi-algebraic set $S$ by $S_{\varepsilon}$ with $\varepsilon=0.01$,
$m=1$, $\mathcal{K}_{1}=\emptyset$, $\mathcal{K}_{2}=\{1,2\}$,
$\ell_{1}=\ell_{2}=0$ and $u_{1}=u_{2}=1$. Only $\omega_{1}$, $z_{1}$
and $z_{2}$ are exhibited in the graphic, since $\omega_{2}=1-\omega_{1}$.
\label{fig:approximation-of-semi-algebraic-set}}
\end{figure}

The inner maximization problems of (\ref{eq:Chebychev_center_optimization})
and (\ref{eq:approx_Chebyshev_center_optimization}) are polynomial
optimization problems (POP). Such problems are ubiquitous and are
encountered in several fields \citep{LHZ:12}, such as: finance \citep{Mar:52,JR:06,KPR:09},
robust and nonlinear control \citep{RN:88,HL:04}, signal processing
\citep{MLD:03,QT:03}, quantum physics \citep{DLMO:08} and materials
science \citep{SYC:08}. It is known that this problem is NP-Hard
in general \citep{MK:87}, but despite this, it is possible to approximate
a POP by convex optimization problems which are computationally feasible
using Lasserre's hierarchy of semidefinite programming relaxations
\citep{Las:01,Las:09}. In fact, by using the GloptiPoly 3 package\footnote{Available for download at \url{http://homepages.laas.fr/henrion/software/gloptipoly3/}.}
\citep{HLL:09} or the SparsePOP package\footnote{Available for download at \url{http://sourceforge.net/projects/sparsepop/}.}
\citep{WKKMS:08} for MATLAB, or ncpol2sdpa library\footnote{Available for download at \url{https://gitlab.com/peterwittek/ncpol2sdpa}.}
\citep{Wit:15} for Python, these relaxations can be easily implemented
and enables the user to transparently construct an increasing sequence
of convex LMI relaxations whose optima are guaranteed to converge
monotonically to the global optimum of the original non-convex global
optimization problem \citep{HL:04}. Moreover, it is possible to numerically
certify the global optimality of the problem.

A direct approach to solve (\ref{eq:Chebychev_center_optimization})
would be to use the mentioned Lasserre's hierarchy to solve the inner
maximization problem for a fixed $\hat{x}=\hat{x}_{0}$ and a local
optimization algorithm to search which $\hat{x}_{0}$ minimizes (\ref{eq:Chebychev_center_optimization}).
This however, besides being a hard numeric problem, would not guarantee
a good approximation to the real value of the Chebychev center. In
the next section alternative ways to compute (\ref{eq:Chebychev_center_optimization})
with different trade-off between precision of the result and speed
of algorithm are proposed.

\section{Computation of robust sigma points\label{sec:Computation-of-robust-sigma-points}}

\subsection{Computation of an outer box to approximate the Chebychev center\label{subsec:Outer-box}}

The first proposed approach to find an approximate solution to (\ref{eq:approx_Chebyshev_center_optimization})
is to compute an outer-bounding box $B$ of $S_{\varepsilon}$ and
approximate the Chebychev center of $S_{\varepsilon}$ by the Chebychev
center of $B$. By Theorem~\ref{thm:semi-algebraic_set_is_bounded},
for $m\ge2$ the constraint set $S_{\varepsilon}$ is bounded. Thus,
each one of the following polynomial optimization problems on $x=(x_{1},\ldots,x_{2m})\in S_{\varepsilon}$
has solution
\begin{align}
\underline{z}_{i} & =\arg\min_{x\in\mathbb{R}^{2m+2}}x_{i}\text{ s.t. }x\in S_{\varepsilon},\ i=1,\ldots,m+1,\nonumber \\
\underbar{\ensuremath{\omega}}_{i} & =\arg\min_{x\in\mathbb{R}^{2m+2}}x_{i+m+1}\text{ s.t. }x\in S_{\varepsilon},\ i=1,\ldots,m+1,\nonumber \\
\overline{z}_{i} & =\arg\max_{x\in\mathbb{R}^{2m+2}}x_{i}\text{ s.t. }x\in S_{\varepsilon},\ i=1,\ldots,m+1,\nonumber \\
\overline{\omega}_{i} & =\arg\max_{x\in\mathbb{R}^{2m+2}}x_{i+m+1}\text{ s.t. }x\in S_{\varepsilon},\ i=1,\ldots,m+1,\label{eq:outer_box_optimization}
\end{align}
and can be used to construct an outer box 
\begin{gather*}
B=\{(z_{1},\ldots,z_{m+1},\omega_{1},\ldots,\omega_{m+1})\in\mathbb{R}^{2m+2}:\underline{z}_{i}\le z_{i}\le\overline{z}_{i},\\
\underline{\omega}_{i}\le\omega_{i}\le\overline{\omega}_{i},i=1,\ldots,m+1\}
\end{gather*}
such that $S_{\varepsilon}\subset B$.  The next theorem gives an
estimate of the error of the outer-bounding box approximation. 
\begin{thm}
Let $c_{B}$ be the center of Chebychev of $B$ and let $c_{S}$ be
the center of Chebychev of $S_{\varepsilon}$. If $d\coloneqq\|c_{B}-c_{S}\|$
is the defined as the distance between the centers of Chebychev, then
\[
d\le\frac{\textrm{diam}(B)}{2},
\]
 where $\textrm{diam}(B)$ is the diameter of $B$, that is, the least
upper bound of the set of all distances between pairs of points in
$B$. 
\end{thm}
\begin{proof}
Suppose that $d>\textrm{diam}(B)/2$. Thus $c_{S}\not\in B$, and
since $c_{S}$ is the center of Chebychev of $S_{\varepsilon}$, one
has that $c_{S}$ is in the convex hull of $S_{\varepsilon}$. But
this implies that $c_{S}$ is in $B$, and the result follows by contradiction. 
\end{proof}

\subsection{Polynomial optimization program to compute minimum enclosing ball
\label{subsec:Two-stage-relaxation} }

Another approximate solution to problem~(\ref{eq:approx_Chebyshev_center_optimization})
can be found by using the two-stage approach of \citep{CLPR:14}.
Since the Chebychev center of $S_{\varepsilon}$ is always inside
of the outer box $B$  computed in Section~\ref{subsec:Outer-box},
problem~(\ref{eq:Chebychev_center_optimization}) is equivalent to
\begin{equation}
(z_{1}^{C},\ldots,z_{m+1}^{C},\omega_{1}^{C},\ldots,\omega_{m+1}^{C})=\underset{\hat{x}\in B}{\arg\min}\ \underset{x\in S_{\varepsilon}}{\max}\|x-\hat{x}\|^{2}.\label{eq:Chebychev_center_optimization_2}
\end{equation}

In the first stage, the inner optimization problem of (\ref{eq:Chebychev_center_optimization_2}),
namely,
\[
J(\hat{x})\coloneqq\max_{x\in S_{\varepsilon}}\|x-\hat{x}\|^{2},
\]
is approximated by a polynomial function $\tilde{J}_{\tau}(\hat{x})$
of degree $2\tau$ using the semidefinite optimization program outlined
in \citep{CLPR:14}. Then, in the second stage, the outer minimization
problem is replaced with the polynomial optimization problem
\begin{equation}
(z_{1}^{C},\ldots,z_{m+1}^{C},\omega_{1}^{C},\ldots,\omega_{m+1}^{C})=\underset{\hat{x}\in B}{\arg\min}\ \tilde{J}_{\tau}(\hat{x}),\label{eq:Chebychev_center_relaxed}
\end{equation}
which can be solved using the Lasserre's hierarchy. As the degree
$2\tau$ of the approximation polynomial increases, the solution of
problem (\ref{eq:Chebychev_center_relaxed}) converges to the solution
of problem (\ref{eq:Chebychev_center_optimization_2}) in the sense
of \citep{CLPR:14}. Alternatively, problem (\ref{eq:approx_Chebyshev_center_optimization})
can be shown to be equivalent to the following polynomial optimization
problem:
\begin{equation}
\begin{array}{cc}
\min_{\hat{x},r}r & \text{ s.t. }\|x-\hat{x}\|^{2}\le r,\\
 & x\in S_{\varepsilon},\,\hat{x}\in B,
\end{array}\label{eq:Chebychev_poly_prog}
\end{equation}
where $\hat{x}$ is the Chebychev center of $S_{\varepsilon}$. In
other words, the Chebychev center is the center of the radius of the
minimum volume ball that encloses $S_{\varepsilon}$. 
\begin{rem}
It is interesting to note that if the gap between $\ell_{i}$ and
$u_{i}$ is not large enough, the resulting semi-definite programs
from relaxing the polynomial optimization problems (\ref{eq:outer_box_optimization})
and (\ref{eq:Chebychev_poly_prog}) can be numerically unstable. In
this case, however, one may use the naive approach of computing sigma
points by using the arithmetic mean of the moments without loss, since
the difference between the real Chebychev center and the point computed
by this method would be not so great due to the small difference between
the upper and lower bounds of the moments.
\end{rem}

\section{Computation of UT transform\label{sec:Computation-of-UT}}

Suppose that $x_{CB}\coloneqq(z_{1}^{C},\ldots,z_{m+1}^{C},\omega_{1}^{C},\ldots,\omega_{m+1}^{C})$
is the Chebychev center of $S$ computed by one of the methods of
Section~\ref{sec:Computation-of-robust-sigma-points}. Based on (\ref{eq:posterior_discrete}),
define the function 
\[
UT_{f}(z_{1},\ldots,z_{m+1},w_{1},\ldots,w_{m+1})\coloneqq\sum_{i=1}^{m+1}w_{i}f(z_{i}).
\]
As discussed above, $UT_{f}$ is a good approximation for $E\{f(X)\}$
for a sufficiently large $m$. While it is true that $x_{CB}$ approximates
of the center of Chebychev of $S$, one also wish to know if $UT_{f}(x_{CB})$
is near to the Chebychev center of $UT_{f}(S)$ (that is, the image
of the set $S$ by the function $UT_{f}$). A class of functions $f$
such that $UT_{f}(x_{CB})$ is near the Chebychev center of $UT_{f}(S)$
is given by the functions such that the solution of the following
optimization problem
\begin{flalign}
\min\,D\ \text{s.t.\ } & (1-D)\|x-y\|\le\|UT_{f}(x)-UT_{f}(y)\|\label{eq:low-distortion-opt-problem}\\
 & \le(1+D)\|x-y\|,\nonumber \\
 & x,y\in S\nonumber 
\end{flalign}
is sufficiently small. If these values are sufficiently small, the
function $UT_{f}$ is a low-distortion geometric embedding \citep{Indy:01},
and this implies that $UT_{f}(x_{CB})$ is near the Chebychev center
of $UT_{f}(S)$. Finally, to compute (\ref{eq:low-distortion-opt-problem}),
one can estimate a solution for (\ref{eq:low-distortion-opt-problem})
by uniformly sampling random values $(x_{i},y_{i})$ of $S$ and computing
$\max\,D_{i}$, where 
\begin{flalign}
D_{i}\coloneqq\min\,D\ \text{s.t.\ } & (1-D)\|x_{i}-y_{i}\|\le\|UT_{f}(x_{i})-UT_{f}(y_{i})\|\label{eq:low-distortion-opt-problem-relaxed}\\
 & \le(1+D)\|x_{i}-y_{i}\|.\nonumber 
\end{flalign}

\section{Numerical experiments}

In this example, the computation of Chebychev center using the methods
proposed in this paper will be illustrated. Consider that one desires
to compute $2$ sigma points in a scenario where the values of the
first and second moment are not precisely known, but it is known that
$E\{X\}\in[0,1]$ and that $E\{X^{2}\}\in[0,1]$. In this case, one
has (\ref{eq:moment_matching_equation_with_uncertainty}) with $m=1$,
$\mathcal{K}_{1}=\emptyset$, $\mathcal{K}_{2}=\{1,2\}$, and $\ell_{1}=-3$
and $u_{1}=4$, and $\ell_{2}=0$ and $u_{2}=5$.

\textbf{Naive method:} For fixed values of $E\{X\}=\mu$ and $E\{X^{2}\}=V$,
one can compute a point inside the set $S$ by using the canonical
formula 
\begin{equation}
\begin{array}{c}
\omega_{1}=0.5,\ \omega_{2}=0.5,\\
z_{1}=\mu+\sqrt{V-\mu^{2}},\ z_{2}=\mu-\sqrt{V-\mu^{2}}.
\end{array}\label{eq:UT_formula_for_m=00003D1}
\end{equation}
A naive choice of sigma points would be the use of the arithmetic
mean of the lower and upper bounds for the moments in (\ref{eq:UT_formula_for_m=00003D1}).
Using (\ref{eq:UT_formula_for_m=00003D1}) with $\mu=\frac{1}{2}(\ell_{1}+u_{1})=0.5$
and $V=\frac{1}{2}(\ell_{2}+u_{2})=2.5$ results in the sigma points
$z_{1}=2$, $z_{2}=-1$ with weights given by $w_{1}=0.5$ and $w_{2}=0.5$. 

\textbf{Method 1: }By using the method proposed in Section~\ref{subsec:Outer-box}
it is possible to compute an outer-bounding box approximation to the
semi-algebraic set $S_{\varepsilon}$. Using $\varepsilon=0.01$,
the computation of an approximate point of the Chebychev center using
outer-bounding box approximation results in the sigma points $z_{1}=0$,
$z_{2}=0$ with weights respectively given by $w_{1}=0.5$ and $w_{2}=0.5$. 

\textbf{Method 2:} Finally, by computing a point using the method
proposed in Section~\ref{subsec:Two-stage-relaxation} with $\varepsilon=0.01$
results in the sigma points $z_{1}=-0.0001$, $z_{2}=0.0001$ with
weights respectively given by $w_{1}=0.1$ and $w_{2}=0.9$. 

In both Method 1 and 2, the MATLAB toolbox Gloptipoly 3 was used to
relax the polynomial optimization problems (\ref{eq:outer_box_optimization})
and (\ref{eq:Chebychev_poly_prog}). A relaxation order of $2$ was
used and the global optimality of the problems was numerically certified
by Gloptipoly3. Moreover, higher relaxation orders could not be used,
since using bigger relaxation orders results in semi-definite programs
with a large number of variables and constraints, and thus yields
an unstable numeric problem. 

Figure~\ref{fig:sigma_points} illustrates the semi-algebraic variety
$S_{\varepsilon}$ and the coordinates of the sigma points calculated
by the three methods. It can be seen by Figure~\ref{fig:sigma_points}
that the point computed by using the outer-bounding box approximation
is the best approximation to the real Chebychev center of the variety.
Moreover, it is important to note that the point calculated by Method
2 is very far from the real Chebychev center of the set $S_{\varepsilon}$.
This is due to the semi-definite relaxation of (\ref{eq:Chebychev_poly_prog})
being far from the exact solution. In principle, one can increase
the relaxation, but this can lead to an unstable numeric problem,
which global optimality cannot be more certified. 

\begin{figure}
\centering{}\includegraphics[scale=0.23]{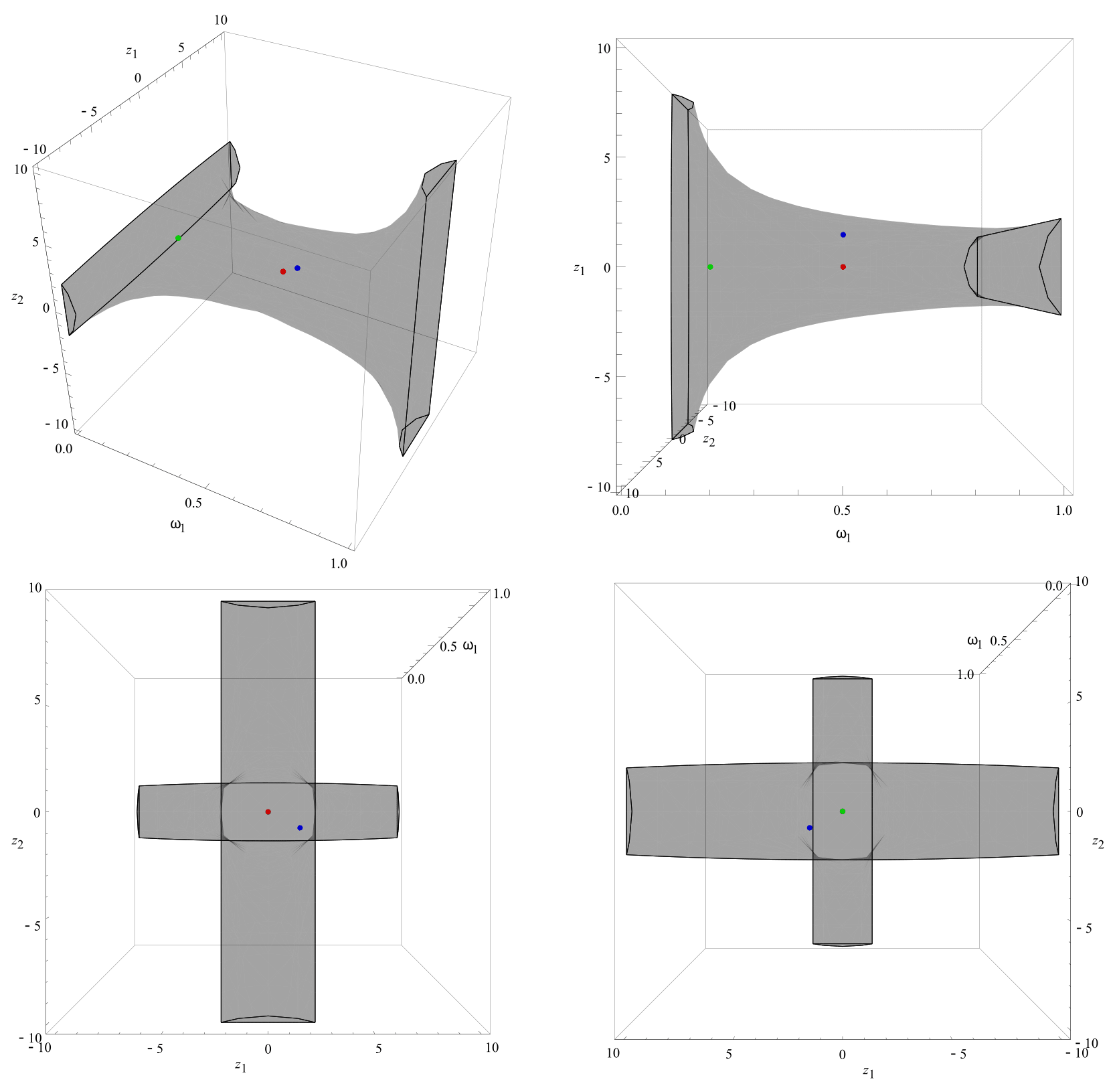}\caption{Semi-algebraic variety $S_{\varepsilon}$ with $\varepsilon=0.01$,
$m=1$, $\mathcal{K}_{1}=\emptyset$, $\mathcal{K}_{2}=\{1,2\}$,
$\ell_{1}=-3$, $u_{1}=4$, $\ell_{2}=0$ and $u_{2}=5$. The blue
point is the coordinates of the sigma-point obtained by the naive
method, the red point is the coordinates of the sigma-point obtained
by the outer-bounding box approximation, and the green point is obtained
by solving the polynomial optimization problem (\ref{eq:Chebychev_poly_prog})
directly. Only $\omega_{1}$, $z_{1}$ and $z_{2}$ are illustrated
in the graphic, since $\omega_{2}=1-\omega_{1}$. \label{fig:sigma_points}}
\end{figure}

Finally, to illustrate the robustness of the computed point, $100$
random samples from $E\{X\}$ and $E\{X^{2}\}$ are respectively draw
from the uniform distributions $U(\ell_{1},u_{1})$ and $U(\ell_{2},u_{2})$.
Solving (\ref{eq:low-distortion-opt-problem-relaxed}) for $f(x)=\sin(x)$
with $500$ samples gives an estimate for $D$ of $0.9956$, and the
posterior expectation in (\ref{eq:posterior_distribution}) is approximated
as 
\[
\textrm{E}\left\{ \sin(X)\right\} \approx\omega_{1}\sin(z_{1})+\omega_{2}\sin(z_{2}),
\]
where $\omega_{1}$, $\omega_{2}$, $z_{1}$ and $z_{2}$ are computed
according to the above methods. The mean error $\left|E\{\sin(X)\}-\sum_{i=1}^{2}\omega_{i}\sin(z_{i})\right|$
by the naive method is $0.0339$, while by Method~1 is $3.6932\cdot10^{-6}$
and by Method~2 is $1.2085\cdot10^{-4}$.

\section{Conclusion\label{sec:Conclusion}}

In this paper, it was proposed a way to devise a robust unscented
transform by the computation of the Chebychev center of the semialgebraic
set defined by the possible choices of sigma points and its weights.
Although, in general, this problem is NP-Hard, some methods are proposed
in this paper to approximate the solution of the original problem
by convex optimization problems. Further works intend to generalize
the present work to higher dimensions, as this could enable novel
filter designs for multivariate dynamical systems. Another possible
extension to this work is to consider adjustments of the proposed
algorithm to work with confidence bounds for the moments instead of
absolute upper and lower bounds, since the former is more common to
be used on interval estimation techniques, such as the bootstrapping
method.

Finally, one limitation of the method is that a great number of sigma
points can lead to a numerically unstable relaxation of the polynomial
optimization problems. Nevertheless, recent advances in polynomial
optimization techniques such as novel relaxation hierarchies based
on linear programming \citep{AM:17}  could help to scale the presented
approaches to a higher number of sigma points. 

\section*{Acknowledgments}

The authors would like to thanks Dario Piga  and Henrique M. Menegaz
for the worthy suggestions and discussions relevant to this paper.
We would also like to thanks the Brazilian agencies CNPq and CAPES
which partially supported this work.

\bibliographystyle{elsarticle-num}

\end{document}